\documentclass[11pt]{amsart}

\usepackage{amssymb}

\newcommand{\N}{{\mathbb N}}

\newcommand{\R}{{\mathbb R}}
\newcommand{\E}{{\mathbb E}}

\newcommand{\C}{{\mathbb C}}
\def\lra{\longrightarrow}

\def\ra{\rightarrow}

\newtheorem{Thm}{Theorem}[section]
\parindent=1em\newtheorem{thm}[Thm]{Theorem}
\newtheorem{coro}[Thm]{Corollary}

\newtheorem{lem}[Thm]{Lemma}
\newtheorem{prop}[Thm]{Proposition}
\newtheorem{defi}[Thm]{Definition}

\begin{document} 
\title[multifractal analysis  of the  branching random walk in $\R^d$]{On the   multifractal analysis  of the  branching random walk in $\R^d$}

\author{Najmeddine Attia}

\address{INRIA Paris-Rocquencourt, domaine de Volucceau, BP 105, 78153 Le Chesnay Cedex, France }
\email{najmeddine.attia@inria.fr }
%\address{ E-mail adress :\text{ najmeddine.attia@inria.fr }}

\maketitle

\begin{abstract}

We establish the almost sure validity of the multifractal formalism for $\R^d$-valued branching random walks on  the whole relative interior of the natural convex domain of study.  
\end{abstract}

\section{Introduction and statement of the result}
\label{S1}
This paper deals with the multifractal analysis of $\R^d$-valued branching random walks. The case $d=1$ is now well known, but it turns out that extending the known results to higher dimensions is not a direct  application of the method used in dimension 1. Let us start with the setting of the problem. 

\medskip

Let $\big( N, X_1, X_2, \cdots  \big)$ be a random vector taking values in $\N_+ \times (\R^d)^{\N_+}$. Then consider  $\big\{ \big(N_{u}, X_{u1}, X_{u2}, \cdots \big)\big\}_{u\in \bigcup_{n\geq 0} \N^n_+}$ be a family of independent copies of the vector $\big( N, X_1, X_2, \cdots \big)$  indexed by the set of finite words over the alphabet $\N_+$ ($\N^0_+$ contains the empty word  denoted by $\emptyset$). Let $T$ be the Galton-Watson tree with defining elements $\{N_u\}$:  we have $\emptyset\in T$ and, if $ u \in T $ and $i \in \N_+$  then $ui$, the concatenation of $u$ and $i$, belongs to $T$ if and only if $1 \leq i \leq N_u$. Similarly, for each $u \in\bigcup_{n\geq 0}  \N^n_+$, denote by $T(u)$ the Galton-Watson tree rooted at $u$ and defined by the $ N_{uv}$, $v \in \bigcup_{n\geq 0} \N^n_+$. For $n\ge 1$ and $u \in\bigcup_{n\geq 0}  \N^n_+$, denote $T(u)\cap  \N^n_+$ by $T_n(u)$. 

\medskip

We assume that $\E(N) > 1$ so that the Galton-Watson tree is supercritical. Without loss of generality, we also assume that the probability of extinction is equal to $0$, so that $\mathbb{P}(N\ge 1) = 1$.\\

For each infinite word $t=t_1 t_2 \cdots\in\N_+^{\N_+}$ and $n\ge 0$, we set $t_{|n}=t_1 \cdots t_n \in\N_+^n$. If $u\in \N^n_+$ for some $n\ge 0$, then $n$ is the length of $u$ and it is denoted by $|u|$ ($t_{|0}=\emptyset$). Then, we denote by $[u]$ the set of infinite words $t \in\N_+^{\N_+}$ such that $t_{||u|}=u$. 

The set $\N_+^{\N_+}$ is endowed with the standard ultrametric distance
 $$ d : (u, v) \mapsto e^ {-\sup \{ |w| : u \in [w],  v \in [w]\}} ,$$
with the convention $\exp(-\infty)=0$. The boundary of the Galton-Watson tree $T$ is defined as the compact set 
$$\partial T = \bigcap_{n \geq 1}  \bigcup_{u \in T_n}  [u] ,$$
consisting of the infinite words $t=t_1 t_2 \cdots$ over $\N_+$ such that for all $n \geq 0, t_{|n}=t_1 \cdots t_n \in T$.  

\medskip

After the strong law of large numbers, we know that, given $t \in \partial T$, we have, if the components of $X$  are integrable and i.i.d., $\displaystyle\lim_{n\ra \infty} \frac{1}{n} S_n (t) = \E(X)$  almost surely, where $S_n (t)= \displaystyle\sum_{k=1}^{n} X_{t_1 \cdots t_k}$.  Since $\partial T$ is not countable,  the following  question naturally arises : are there some $t\in \partial T$ so that  $\displaystyle\lim_{n\ra \infty} \frac{1}{n} S_n (t) = \alpha \neq \E(X)?$ Multifractal analysis is a framework adapted to answer this question. Consider the set $\mathcal I$ of those   $\alpha \in \R^d$ such that 
$$
E (\alpha) = \Big\{t\in \partial T : \displaystyle\lim_{n \ra \infty} \frac{1}{n} \sum_{k=1}^{n} X_{u_1\cdots u_k } =\alpha  \Big\} \;\;\neq \emptyset.$$
These level sets can be described geometrically through their Hausdorff dimensions. They have been studied by many authors when $d=1$, see for instance  \cite{HoWa,Falc,Mol,jul3,Big3}; all these papers also deal with the multifractal analysis of associated Mandelbrot measures (see also \cite{KP,Pey2,LiuRouault} for the study of Mandelbrot measures dimension). \\

The vector space $\R^d$ is endowed with the canonical scalar product and the associated euclidean norm respectively denoted $\left\langle \cdot |\cdot  \right\rangle$ and $\left\|\cdot  \right\|$. For all $x\in \R^d$ and $r\ge 0$, $B(x,r)$ stands for the closed Euclidean ball of radius $r$ centered at $x$. 

\medskip

We will state our main result by using the notion of multifractal formalism (see~\cite{Pey4} for an abstract  vectorial multifractal formalism). Let us define the pressure like function 
 $$P(q)=\displaystyle\limsup_{n\rightarrow\infty}\frac{1}{n}\log \Big(\displaystyle\sum_{u\in T_n}  \exp \big( \left\langle q |S_n(u)  \right\rangle\big) \Big)\;\;\;\; (q\in \R^d).$$
Let  $P^*$ stand for the Legendre transform of the function $P$, where by convention the Legendre transform of a mapping $f:\R^d\lra \R$ is  defined as the concave and upper semi-continuous function : $$f^*(\alpha):=\displaystyle\inf_{q\in\R^d} \Big(f(q)-\left\langle q,\alpha\right\rangle \Big). $$
We say that the multifractal formalism holds at $\alpha\in\R^d$ if $\dim\, E(\alpha)=P^*(\alpha)$. \\

\medskip

%Let   $$ S(q) = \displaystyle\sum_{i=1}^N \exp \big(\left\langle q|X_i \right\rangle\big),\;\;\;(q\in \R^d). $$
For the sake of simplicity we will assume throughout that the logarithmic moment generating function
$$ \widetilde P (q)  : q\in \R^d \mapsto \log \E\Big(\sum_{i=1}^N \exp \big(\left\langle q|X_i \right\rangle\big)\Big),$$ 
is finite over $\R^d$ (see Section~\ref{remarks} for the relaxation of this assumption). 

\medskip

Let $$J = \Big\{ q \in \R^d  ; \tilde{P}(q)-\langle q|\nabla \tilde{P}(q)\rangle   > 0 \Big\}.$$

\medskip

Let
$$ \Omega_{\gamma}^1= \mathrm{int} \Big\{ q : \E [ \big | \displaystyle\sum_{i=1}^N e^{\left\langle q| X_i  \right\rangle  } \big |^\gamma ] < \infty \Big\}, \; \;\; \Omega^1 = \bigcup_{\gamma\in (1, 2]} \Omega_{\gamma}^1,$$
and $${\mathcal{J}} = J \cap \Omega^1
  \;\; \text {and} \;\;  I = \Big\{ \nabla\widetilde P(q)  ;q\in {\mathcal J}  \Big\}.$$

Our main result is the following. 
\begin{thm}\label{tt} Suppose that $\widetilde P$ is finite over $\R^d$. With probability $1$, for all $  \alpha\in I$, we have $\widetilde P^*(\alpha)=P^*(\alpha)$ and the multifractal formalism holds at $\alpha$, i.e., 
$\dim E (\alpha)=\widetilde P^*(\alpha)$; in particular, $E(\alpha) \neq \emptyset$.
\end{thm}
In dimension 1, this result has been proved  when $N$ is not random in \cite{jul3},  and in the weaker form, for each fixed $\alpha \in I$, almost surely $\dim E (\alpha)=\widetilde P^*(\alpha)$, when $N$ is random in \cite{HoWa,Falc,Mol,Big3}. Further comments on this result and its possible improvements are given in Section~\ref{remarks}.
  \medskip

%%%%%%%%%%%%%%%%%%%%%%%%%%%%%%%%%%%%%%%%%%%%%%%%%%%%%%%%%%%%%%%%%%%%%%%%%%%%%%%%%%%%%%%%%%%%%%%%%%%%%%%%%%%%%%%%%%%%%%%%%%%%%%%%%%
%%%%%%%%%%%%%%%%%%%%%%%%%%%%%%%%%%%%%%%%%%%%%%%%%%%%%%%%%%%%%%%%%%%%%%%%%%%%%%%%%%%%%%%%%%%%%%%%%%%%%%%%%%%%%%%%%%%%%%%%%%%%%%%%%%
%%%%%%%%%%%%%%%%%%%%%%%%%%%%%%%%%%%%%%%%%%%%%%%%%%%%%%%%%%%%%%%%%%%%%%%%%%%%%%%%%%%%%%%%%%%%%%%%%%%%%%%%%%%%%%%%%%%%%%%%%%%%%%%%%%%

\section{Proof}
\label{S2}
%%%%%%%%%%%%%%%%%%%%%%%%%%%%%%%%%%%%%%%%%%%%%%%%%%%%%%%%%%%%%%%%%%%%%%%%%%%%%%%%%%%%%%%%%%%%%%%%%%%%%%%%%%%%%%%%%%%%%%%%%%%%%%%%%%%
                     \subsection{Upper bounds for the Hausdorff dimension }
%%%%%%%%%%%%%%%%%%%%%%%%%%%%%%%%%%%%%%%%%%%%%%%%%%%%%%%%%%%%%%%%%%%%%%%%%%%%%%%%%%%%%%%%%%%%%%%%%%%%%%%%%%%%%%%%%%%%%%%%%%%%%%%%%%%                 

\begin{prop} \label{p1}
With probability $1$, $P(q)\leq \widetilde P (q)$ for all $q\in \R^d$, and then $ P^*(\alpha)\leq \widetilde P^*(\alpha)$, for all $\alpha \in \R^d.$
\end{prop}
\begin{proof}The  functions $\tilde{P}$ and  $P$ being convex and thus continuous, we only need to prove the inequality $P(q)\leq \widetilde P (q)$ for each $q\in \R^d$ almost surely. Fix $q\in \R^d$. For $s > \widetilde P(q) $ we have 
\begin{eqnarray*}
\E \Big (\sum_{n\geq1}e^{-ns}\sum_{u\in T_n} \exp \big( \left\langle q |S_{n}(u)\right\rangle\big) \Big )&=&\sum_{n\geq1}e^{-ns} \E \Big (\sum_{i=1}^{N}  \exp \big(\left\langle  q | X_i \right\rangle \big)\Big )^n \\
&=&\sum_{n\geq1}e^{n(\tilde{P}(q)-s)}.
\end{eqnarray*}
Consequently,
$\displaystyle\sum_{n\geq1} e^{-ns}\displaystyle\sum_{u\in T_n} \exp ( \left\langle q |S_{n}(u)\right\rangle) <\infty $ almost surely, so that we have $\displaystyle\sum_{u\in T_n} \exp ( \left\langle q |S_{n}(u)\right\rangle)=O(e^{ns})$ and $P(q) \le s$. Since $s >\tilde{P}(q)$ is arbitrary, we have the conclusion.
\end{proof}
\begin{prop}\label{p2}
With probability $1$, for all $\alpha \in \R^d$, $\dim E(\alpha)\leq P^*(\alpha)$, a negative dimension meaning that $E(\alpha)$ is empty.
\end{prop}
\begin{proof} We have 
 \begin{eqnarray*}
E(\alpha) &=&  \bigcap _{\epsilon > 0} \bigcup_{N\in \N^*} \bigcap_{n\geq N} \Big\{t\in \partial T ;  \|  S_n(t) -  n \alpha \| \leq n\epsilon \Big\} \\ 
&\subset &  \bigcap_{q\in \R^d} \bigcap _{\epsilon > 0} \bigcup_{N\in \N^*} \bigcap_{n\geq N} \Big\{t\in \partial T ;  \; \left|  \left\langle q | S_n(t) -  n \alpha  \right\rangle \right| \leq n \|q\| \epsilon \Big\}.
\end{eqnarray*}
Fix $q\in \R^d$ and $\epsilon >0$. For $N \geq 1$, the set $E(q, N, \epsilon, \alpha)=  \bigcap_{n\geq N} \big\{t\in \partial T ;  \; \left|  \left\langle q | S_n(t) -  n \alpha  \right\rangle \right| \leq n \|q\| \epsilon \big\}$ is covered by the union of those $[u]$ such that $u \in T_n$, $n\ge N$, and
 $\left\langle q | S_n(u)-n\alpha\right\rangle+n\|q\| \epsilon\geq 0$. \\
 We define the s-dimensional Hausdorff measure of a set $E$ by 
 $${\mathcal H}^s(E) = \lim_{\delta\to 0} {\mathcal H}_{\delta}^s (E)=\lim_{\delta\to 0} \inf\Big\{\sum_{i\in\N}  \mathrm{diam}(U_i)\Big\},
$$
the infimum being taken over all the countable coverings $(U_i)_{i\in\N}$ of $E$  of diameters less than or equal to $\delta$.\\
 Thus, for $s\geq 0$ and $n \ge N$,
$${\mathcal H}_{e^{-n}}^s \big(E(q, N, \epsilon, \alpha) \big) \leq \displaystyle\sum_{u\in T_n} e^{-ns} \exp \big(\left\langle q | S_n(u)-n\alpha\right\rangle+n\|q\| \epsilon \big).$$
Consequently, if $\eta >0$ and $s > P(q) + \eta - \left\langle q | \alpha\right\rangle+\|q\| \epsilon$, by definition of $P(q)$, for $N$ large enough we have 
$${\mathcal H}_{e^{-n}}^s \big( E(q, N, \epsilon, \alpha) \big) \leq  e^{-n\eta/2}.$$
This yield ${\mathcal H}^s \big( E(q, N, \epsilon, \alpha) \big)=0$,  hence $\dim E(q, N, \epsilon, \alpha) \leq s.$ Since this holds for all $\eta > 0$ we get $\dim E(q, N, \epsilon, \alpha) \leq P(q) - \left\langle q | \alpha\right\rangle+\|q\| \epsilon$. It follows that 
$$\dim E(\alpha)\leq \displaystyle\inf_{q\in \R^d} \inf_{\epsilon >0} \sup_{N\in \N^*} P(q) - \left\langle q | \alpha\right\rangle+\|q\| \epsilon = P^*(\alpha).$$
If $P^*(\alpha) <0$, we necessarily have $E(\alpha)= \emptyset$. 
\end{proof}

%%%%%%%%%%%%%%%%%%%%%%%%%%%%%%%%%%%%%%%%%%%%%%%%%%%%%%%%%%%%%%%%%%%%%%%%%%%%%%%%%%%%%%%%%%%%%%%%%%%%%%%%%%%%%%%%%%%%%%%%%%%%%%%%%%%
                  \subsection{Lower bounds for the Hausdorff dimensions}
%%%%%%%%%%%%%%%%%%%%%%%%%%%%%%%%%%%%%%%%%%%%%%%%%%%%%%%%%%%%%%%%%%%%%%%%%%%%%%%%%%%%%%%%%%%%%%%%%%%%%%%%%%%%%%%%%%%%%%%%%%%%%%%%%%%
For   $(q,p) \in  {\mathcal J} \times  [1, \infty)$, we define the function
 $$ \phi (p,q) = e^{\tilde{P}(pq)-p\tilde{P}(q)}. $$  
 and for $q \in  {\mathcal J} $ and $u\in T$ , we define the sequence 
$$Y_n(u,q)=  \E \Big(\displaystyle\sum_{i=1}^N e^{  \left\langle q | X_i \right\rangle} \Big)^{-n}\displaystyle\sum_{v\in T_n(u)} e^{ \left\langle q| S_{|u|+n}(uv)-S_{|u|}(u)\right\rangle},\quad (n\ge 1).$$
When $u = \emptyset$, $Y_n(\emptyset,q)$ will be denoted by $Y_n(q)$. \\
 The sequence $ \big( Y_n(u,q) \big)_{n\geq 1}$ is a positive martingale with expectation $1$, which converges almost surely and in $L^1$ norm to a positive random variable  $Y(u,q)$ (see \cite{KP,Big1} or \cite[Theorem 1]{Big2}). However, our study will need the almost sure simultaneous convergence of  these martingales to positive limits (see Proposition~\ref{pp1}(1)). 
\medskip

Let us  state two propositions, the proof of which is postponed to the end of this section. The uniform convergence part  of  Proposition \ref{pp1} is essentially Theorem 2 of \cite{Big2}, with slightly different assumptions. However, for the reader's convenience, and since the method used by Biggins will be used also in proving Propositions~\ref{pp2} and~\ref{pp3}, we will include its proof. The second part of  Proposition \ref{pp1} defines the family of Mandelbrot measures built simultaneously to control the Hausdorff dimensions of the sets $E(\nabla P(q))$, $q\in {\mathcal J}$, from below. Then Proposition \ref{pp2} introduces suitable logarithmic moment generating functions associated with these measures to get the desired  lower bounds via large deviations inequalities. 

\begin{prop}\label{pp1} 
\begin{enumerate}
\item  Let $K$ be a compact subset of ${\mathcal J}$. There exists  $p_K\in (1,2]$ such that  for all $u \in  \bigcup_{n\ge 0}\N_+^n$, the continuous functions $q\in K\mapsto Y_n(u,q)$ converge  uniformly, almost surely  and in $L_{p_K}$ norm, to a limit $q\in K\mapsto Y(u,q)$. In particular,
 $\E(\displaystyle\sup_{q\in K } Y(u,q)^{p_K}) < \infty.$ Moreover, $Y(u,\cdot)$ is positive almost surely.  
 
In addition, for all $n\geq 0$, $\sigma \big (\{ (X_{u1},\cdots,X_{uN(u)}), u\in T_n \}\big )$ and $\sigma \big (\{ Y(u,\cdot), u\in T_{n+1}\} \big )$ are independent, and the random  functions $Y(u,\cdot), u\in T_{n+1}$, are  independent copies of $Y(\cdot)$.\\ 

\item  With  probability  $1$, for all  $q \in {\mathcal J}$, the weights  
$$\mu_q ([u])=\E \Big(\sum_{i=1}^N e^{  \left\langle q | X_i \right\rangle} \Big)^{-|u|}e^{ \left\langle q| S_{|u|}(u)\right\rangle}  Y(u, q)$$
define  a measure on  $\partial T$ . 
\end{enumerate}
\end{prop}
For $q\in {\mathcal J}$, let 
$$L_n(q,\lambda)=\frac{1}{n} \log \int_{\partial T} \exp \big(\left\langle \lambda | S_n(t) \right\rangle \big) d\mu_q(t), \;\;\;(\lambda\in \R^d),$$
and 
$$L(q, \lambda) = \displaystyle\limsup_{n\ra\infty} L_n(q,\lambda). $$
\begin{prop}\label{pp2}
Let $K$ be a compact subset of ${\mathcal J}$. There exists a compact neighborhood $\Lambda$ of the origin such that, with probability $1$,
\begin{equation}\label{eq1}
\displaystyle \lim_{n\ra\infty} \sup_{\lambda \in \Lambda}\sup_{q\in K} |L_n(q,\lambda)-(\widetilde P(q+\lambda)-\widetilde P (q))| = 0, 
\end{equation}
in particular $L(q,\lambda)=\widetilde P(q+\lambda)-\widetilde P (q)$ for $(q,\lambda)\in K\times \Lambda$. 
\end{prop}
%%%%
\begin{coro}\label{cc1}
With probability $1$, for all $q\in {\mathcal J}$, for $\mu_q$-almost every $t\in \partial T$,  
 $$\displaystyle\lim_{n\ra\infty} \frac{S_n(t)}{n} = \nabla\widetilde P(q).$$
\end{coro}

\begin{proof} It follows from Proposition~\ref{pp2} that there exists $\Omega'\subset \Omega$ with $\mathbb P(\Omega')=1$, and  such that for all $\omega\in\Omega'$, for all $q\in {\mathcal J}$, there exists a neighborhood of $0$ over which $L_n(q,\lambda)$ converges uniformly in $\lambda$ towards $L(q,\lambda)=\widetilde P(q+\lambda)-\widetilde P (q)$. 

 For each $\omega\in \Omega'$, let us define for each  $q\in {\mathcal J}$ the  sequence of measures $\{\nu^\omega_{q,n}\}_{n\ge 1}$ as 
 \begin{equation}\label{nuqn}
 \nu^\omega_{q,n}(B)=\mu_q(\{t\in \partial T : \frac{1}{n} S_n(t)  \in B\})
 \end{equation} 
  for all Borel set  $B \subset \R^d$. We denote  $L(q,\lambda)$ by $L_q(\lambda)$. Since 
 $$
 L_n(q,\lambda)=\frac{1}{n}\log\int_{\R^d} \exp(n\langle \lambda|u\rangle)\, d\nu_{q,n}^\omega(u),
 $$
 applying G\"artner-Ellis Theorem \cite[Thm. 2.3.6]{De-Zei},   for all closed subsets $\Gamma$ of $\R^d$, we have for all $q\in {\mathcal J}$
$$\displaystyle\limsup_{n\ra\infty }\frac{1}{n} \log \nu^\omega_{q,n}(\Gamma) \leq \sup_{\alpha\in \Gamma} L_q^*(\alpha).$$                      

Let $\epsilon>0$, and for each $q\in {\mathcal J}$ let  $A_{q,\epsilon} =\big\{\alpha\in\R^d:d(\alpha, \nabla L_q(0)) \ge  \epsilon) \big\}$, where $d$ is a Euclidean distance in $\R^d$. We have  $\displaystyle\limsup_{n\ra\infty }\frac{1}{n} \log \nu_{q,n}^\omega (A_{q,\epsilon})  \leq \sup_{\alpha\in A_{q,\epsilon}} L_q^*(\alpha)$. In addition, since $L_q(\lambda)=\widetilde P(q+\lambda)-\widetilde P (q)$ in a neighborhood of $0$,  we have $\nabla L_q(0)=\nabla\widetilde P(q)$  and $L_q^* (\nabla L_q(0))=0=\max L_q^*$. Moreover, since $L_q$ is differentiable at~$0$, we have $L_q^*(\alpha)<L_q^*(\nabla L_q(0))$ for all $\alpha\neq \nabla L_q(0)$. Indeed, suppose that $L_q^*(\alpha)=0$; then it follows from the definition of the Legendre transformation and the fact that $L_q(0)=0$, that $$\forall \lambda \in \R^d, \;\; L_q(\lambda) \geq L_q(0) + \left\langle \lambda | \alpha \right\rangle,$$
hence $\alpha$ belongs to the  subgradient of $L_q$ at $0$, which from Proposition \ref{pppp1} reduces to $\{\nabla L_q(0)\}$.

Now, due to the upper semi-continuity of the concave function $L_q^*$, we have $ \gamma_{q,\epsilon}=\sup_{\alpha\in A_{q,\epsilon}} L_q^*(\alpha)<0$. 

Consequently,  for all $q\in {\mathcal J}$,  for $n$ large enough, $\nu_{q,n}^\omega (A_{q,\epsilon}) \leq  e^{n\gamma_{q,\epsilon/2}}$, i.e. 
$$\mu_q \big( \big\{t\in \partial T : \frac{1}{n} S_n(t)  \in A_{q,\epsilon} \big\} \big) \le  e^{n\gamma_{q,\epsilon}/2}.$$
Then it follows from the Borel-Cantelli Lemma (applied with respect to $\mu_q$) that for all $q\in {\mathcal J}$, 
for $\mu_q$-almost every $t\in\partial T$, we have $\frac{1}{n} S_n(t) \in B \big(\nabla\widetilde P(q), \epsilon \big) $ for $n$ large enough. Letting $\epsilon$ tend to 0 along a countable sequence yields the desired conclusion. 
\end{proof}
 
\begin{coro}
With probability 1, for all $q\in {\mathcal J}$, the sequence of random measure $(\nu_{q,n}^\omega)_{n\ge 1}$ defined in \eqref{nuqn} satisfies the following large deviation property: for all $\lambda$ in a neighborhood of $0$, 
$$ \displaystyle\lim_{\epsilon \ra 0}\lim_{n\ra\infty} \frac{1}{n} \log \nu_{q,n}^\omega \big(B(\nabla L_q(\lambda),  \epsilon)\big) = L_q^*(\nabla L_q(\lambda)),$$   
where $L(q,\lambda)=\widetilde P(q+\lambda)-\widetilde P (q)$.
\end{coro}

\begin{proof}
It is a consequence of G\"artner-Ellis theorem (see \cite{De-Zei}).
\end{proof}

We need a last proposition to get the lower bounds in Theorem~\ref{tt}. Its proof will end the section. 
\begin{prop}\label{pp3}
With probability $1$, for all $q\in {\mathcal J}$, for $\mu_q$-almost every $t\in \partial T$,
$$\displaystyle\lim_{n\rightarrow\infty}\frac{\log Y(t_{|n},q)}{n} =0.$$ 
\end{prop}

\medskip
\noindent\textbf{Proof of the lower bounds in Theorem~\ref{tt}:} From Corollary \ref{cc1}, we have with probability $1$, $\mu_q \big(E (\nabla\widetilde P(q)) \big)=1$.  In addition, with probability $1$, for $\mu_q$-almost every  $t\in E ( \nabla\widetilde P(q) ) $, from the same  corollary and  Proposition \ref{pp3}, we have 
\begin{eqnarray*}
\lim_{n\ra\infty} \frac{\log( \mu_q[t_{|n}])}{\mathrm{\log (diam}([t_{|n}]))}&=& \lim_{n\ra\infty} \frac{-1}{n}\log\Big(\exp \big(\left\langle q | S_n(t)  \right\rangle -n\widetilde P(q) \big)    Y(t_{|n}, q) \Big) \\
&=& \widetilde P(q)+ \lim_{n\ra\infty} \frac{\left\langle q | S_n(t)  \right\rangle }{-n} -  \lim_{n\ra\infty} \frac{\log  Y(t_{|n}, q) }{n}\\
&=& \widetilde P(q) - \left\langle q| \nabla\tilde{ P}(q) \right\rangle = \widetilde P^*(\nabla\widetilde P(q)). 
\end{eqnarray*} 
We deduce the result from the mass distribution principle (Theorem \ref{Bill}).

\medskip

Now, we  give the proofs of the previous propositions. 
%%%%%%%%%%%%%%%%%%%%%%%%%%%%%%%%%%%%%%%%%%%%%%%%%%%%%%%%%%%%%%%%%%%%%%%%%%%%%%%%%%%%%%%%%%%%%%%%%%%%%%%%%%%%%%%%%%%%%%%%%%%%%%%%%%%%%%%%%%%%%%%%
%%%%%%%%%%%%%%%%%%%%%%%%%%%%%%%%%%%%%%%%%%%%%%%%%%%%%%%%%%%%%%%%%%%%%%%%%%%%%%%%%%%%%%%%%%%%%%%%%%%%%%%%%%%%%%%%%%%%%%%%%%%%%%%%%%%%%%%%%%%%%%%%
%%%%%%%%%%%%%%%%%%%%%%%%%%%%%%%%%%%%%%%%%%%%%%%%%%%%%%%%%%%%%%%%%%%%%%%%%%%%%%%%%%%%%%%%%%%%%%%%%%%%%%%%%%%%%%%%%%%%%%%%%%%%%%%%%%%%%%%%%%%%%%%%
\subsection{Proofs of Propositions~\ref{pp1}, \ref{pp2} and \ref{pp3}}
We start with several lemmas.
\begin {lem}\label{l1} Recall that, for   $(q,p) \in  {\mathcal J}  \times  [1, \infty)$, $ \phi (p,q) = e^{\tilde{P}(pq)-p\tilde{P}(q)}.$ Then, for all nontrivial compact  $K \subset {\mathcal J} $ there exists  a real number  $1 < p_{K} < 2$ such that for all $1<p\le p_K$ we have 
$$\displaystyle\sup_{q \in  K} \phi (p_{K},q) <  1.$$
\end{lem}
\begin{proof} Let  $q \in  {\mathcal J}$,  one has   $\frac{\partial \phi}{\partial p}  (1^+, q) < 0
 $ and there exists  $p_{q}>1$ such that   $\phi (p_{q}, q) < 1$. Therefore, in a neighborhood  $V_{q}$ of $q$,  one has    $\phi (p_q, q') < 1$ for all $q'\in V_q$. If  $K$ is a nontrivial  compact  of  ${\mathcal J}$, it is covered  by a finite number of such $V_{q_{i}}$. Let $p_{K} = \displaystyle \inf_{i} p_{q_{i}}$. If $1<p\le p_K$ and $\sup_{q \in  K} \phi (p,q) \ge 1$, there exists $q\in K$ such that  $\phi (p,q)\ge 1$, and $q\in V_{q_i}$ for some $i$. By log-convexity of the mapping  $p\mapsto  \phi(p,q)$ and the fact that $\phi(1,q)=1$, since $1<p\le p_{q_i}$ we have $\phi (p, q) < 1$, which is a contradiction.  
\end{proof}
\begin{lem}\label{l2}
For all compact $K\subset {\mathcal J}$, there exists $\tilde p_K > 1$ such that, 
$$  \sup_{q\in K}\E \Big( \big(\sum_{i=1}^N e^{\left\langle q| X_i  \right\rangle  } \big)^{\tilde p_K} \Big) < \infty. $$
\end{lem} 
\begin{proof}
Since $K$ is compact and the family of open sets $J\cap \Omega^1_\gamma$ increases to $\mathcal J$ as $\gamma$ decreases to $1$, there exists $\gamma\in (1,2]$ such that $K\subset \Omega_\gamma^1$. Take $\tilde p_K=\gamma$. The conclusion comes from the fact that the function $q\mapsto \E \Big( \big(\sum_{i=1}^N e^{\left\langle q| X_i  \right\rangle  } \big)^{\tilde p_K} \Big)$ is convex over   $ \Omega_{\tilde p_K}^1$ so continuous.

\end{proof} 

\medskip

The next lemma comes from \cite{Big2}.
\begin{lem}\label{ll2}
If $\{X_i\}$ is a family of integrable and independent complex random variables with $\E(X_i) =0$, then $\E |\sum X_i|^p \leq 2^p \sum \E |X_i|^p$ for $1\leq p\leq 2$.
\end{lem}

%%%%%%%%%%%%%%%%%%%%%%%%%%%%%%%%%%%%%%%%%%%%%%%%%%%%%%%%%%%%%%%%%%%%%%%%%%%%%%%%%%%%%%%%%%%%%%%%%%%%%%%%%%%%%%%%%%%%%%%%%%%%%%%%%%%%%%%%%%%%%%%%
%%%%%%%%%%%%%%%%%%%%%%%%%%%%%%%%%%%%%%%%%%%%%%%%%%%%%%%%%%%%%%%%%%%%%%%%%%%%%%%%%%%%%%%%%%%%%%%%%%%%%%%%%%%%%%%%%%%%%%%%%%%%%%%%%%%%%%%%%%%%%%%%
%%%%%%%%%%%%%%%%%%%%%%%%%%%%%%%%%%%%%%%%%%%%%%%%%%%%%%%%%%%%%%%%%%%%%%%%%%%%%%%%%%%%%%%%%%%%%%%%%%%%%%%%%%%%%%%%%%%%%%%%%%%%%%%%%%%%%%%%%%%%%%%%

\begin{lem}\label{ll1}
Let  $(N, V_1,V_2,\cdots )$ be a random vector taking values in $\N_+\times\C^{\N_+}$ and such that $ \sum_{i=1}^N V_i $ is integrable and $\E\big(    \sum_{i=1}^N V_i \big) =1$.  Let $M$  be an integrable  complex random variable. Consider  $ \big\{(N_{u}, V_{u1}, V_{u2},\ldots) \big\}_{u\in \bigcup_{n\ge 0}\N_+^n}$ a sequence of independent  copies of $(N, V_1,\cdots, V_N)$ and $\{M_{u}\}_{u\in \bigcup_{n\ge 0}\N_+^n}$ a sequence of copies of $M$ such that for all $n\ge 1$, the random variables $M(u)$, $u\in \N_+^n$, are independent, and independent of $\big\{ (N_{u}, V_{u1}, V_{u2},\ldots ) \big\}_{u\in \bigcup_{k= 0}^{n-1}\N^k_+}$. We define the sequence $(Z_n)_{n\geq 0}$ by   $Z_0=\E(M)$ and for $n\geq 1$
$$Z_n=\sum_{u\in T_n}\big( \prod_{k=1}^{n} V_{u_{|k}} \big)M(u).$$
Let $p\in (1, 2]$. There exists a constant $C_p $  depending on $p$ only such that for all $n\geq 1$
\begin{equation*}
 \E(|Z_n - Z_{n-1}|^p) \leq  C_p  \E ( |M|^{p}) \Big (\E\big(  \sum_{i=1}^{N}| V_{i} |^{p} \big)\Big )^{n-1} \Big (\E\big(  \sum_{i=1}^{N}| V_{i} |^{p} \big) +   \E \big(  | \sum_{i=1}^{N} V_i |^p\big)+1\Big ).
 \end{equation*}
\end{lem}
\begin{proof} The definition of the process $Z_n$ gives immediately that 
\begin{equation}\label{ee1}
Z_n-Z_{n-1}=\displaystyle\sum_{u\in T_{n-1}} \prod_{k=1}^{n-1} V_{u_{|k}} \Big(\sum_{i=1}^{N_u} V_{ui}M(ui) -M(u)\Big).
\end{equation}
For each $n\ge 1$ let ${\mathcal F}_n=\sigma \big\{(N_u, V_{u1},\ldots) :  |u| \leq n-1 \big\}$ and let $\mathcal {F}_{0}$ be the trivial sigma-field. The random variable $Z_n-Z_{n-1}$ is a weighted sum of independent and  identically distributed random variables with zero mean, namely the random variables  $\sum_{i=1}^{N_u} V_{ui}M(ui) -M(u)$, which are independent of $\mathcal{F}_{n-1}$. Applying the Lemma \ref{ll2} with $X_u = \displaystyle \prod_{k=1}^{n-1} V_{u_{|k}} \Big(\sum_{i=1}^{N_u} V_{ui}M(ui) -M(u)\Big)$, $u\in T_n$, conditionally on  $\mathcal F_{n-1}$, and noticing that the weights $\prod_{k=1}^{n-1} V_{u_{|k}}$, $u\in T_{n-1}$, are $\mathcal F_{n-1}$-measurable, we get
\begin{eqnarray*} 
\E \big(|Z_n - Z_{n-1}|^p\big) &= & \E \Big(\E \big(|Z_n - Z_{n-1}|^{p}  \; |\;{\mathcal F}_{n-1} \big)\Big)\\
& \leq &\E \Big( 2^p \sum_{u\in T_{n-1}} \prod_{k=1}^{n-1} |V_{u_{|k}}|^{p} \E \big | \sum_{i=1}^{N_u} V_{ui} M(ui) -M(u) \big |^{p}\Big).
\end{eqnarray*}
It is easy to see that $\E\big(\displaystyle\sum_{u\in T_{n-1}} \prod_{k=1}^{n-1} |V_{u_{|k}}|^p\big) = \displaystyle\prod_{k=1}^{n-1}\E\big(\sum_{i=1}^N | V_i  |^p\big)$.
Using the inequality
\begin{equation}\label{2eq3}
 |x+y|^r \leq 2^{r-1}(|x|^r + |y|^r), \;\;\;( r > 1),
\end{equation}
we get
\begin{eqnarray*}
\E\big(\big  |\sum_{i=1}^{N_u} V_{ui} M(ui) -M(u) \big |^p \big)& \leq& 2^{p - 1}\E \big (  \big | \displaystyle\sum_{1=1}^{N_u}V_{ui} M(ui) \big  |^{p} + \E(|M |)^{p} \big ).
\end{eqnarray*}
Write $M(ui)= M(ui)-\E(M(ui))+\E(M(ui))$. Then from the inequality (\ref{2eq3}), we get
\begin{eqnarray*}
&&\E\big(\big  | \displaystyle\sum_{i=1}^{N_u} V_{ui} M(ui) \big  |^p \big) =  \E\big ( \big | \displaystyle\sum_{1=1}^{N_u} V_{ui} (M(ui)- \E(M(ui)))+  V_{ui}\E(M(ui)) \big |^{p} \big )  \\
&& \leq 2^{p-1} \E \big( \big | \displaystyle\sum_{i=1}^{N_u} V_{ui} (M(ui)- \E(M(ui))) \big |^p  \big) + 2^{p-1}\E(|M|^p)\E \big( \big | \displaystyle\sum_{1=1}^{N_u}   V_{ui} \big |^p \big).    
\end{eqnarray*}
It follows from the Lemma \ref{ll2} applied  with $X_i=V_{ui} (M(ui)- \E(M(ui)))$ conditionally on $(N_u,V_{u1},\cdots,V_{uN_{u}})$, and from the independence of $M(ui)$ and $(N_u,V_{u1},\cdots,V_{uN_{u}})$, that
\begin{eqnarray*} 
 \E\big( \big| \displaystyle\sum_{1=1}^{N_u} V_{ui} (M(ui)- \E(M(ui))) \big |^p  \big) &\leq& 2^p \E\big(  \displaystyle\sum_{i=1}^{N_u} \big | V_{ui} (M(ui)- \E(M(ui)))\big  |^p  \big)\\
&\leq&2^p \E\big(\big  | M(u)- \E(M(u)) \big |^p  \big) \E\big(  \displaystyle\sum_{i=1}^{N_u}| V_{ui} |^p \big )\\
&\leq & 2^{2p}\E\big( |M|^p \big ) \E\big(  \displaystyle\sum_{i=1}^{N}| V_{i}|^p \big).
\end{eqnarray*} 
Finally, we have 
$$ \E\big( \big | \displaystyle\sum_{i=1}^{N_u} V_{ui}M(ui)-M(u) \big  |^{p} \big) \leq  C_p \E |M|^{p}  \Big (\E\big(  \displaystyle\sum_{i=1}^{N}| V_{i} |^{p} \big) +   \E \big(  |\displaystyle\sum_{i=1}^{N} V_i |^p\big)+1\Big ).
$$
\end{proof}
Now we prove Propositions~\ref{pp1}, \ref{pp2} and \ref{pp3}.
%%%%%%%%%%%%%%%%%%%%%%%%%%%%%%%%%%%%%%%%%%%%%%%%%%%%%%%%%%%%%%%%
%%%%%%%%%%%%%%%%%%%%%%%%%%%%%%%%%%%%%%%%%%%%%%%%%%%%%%%%%%%%%%%%%%%%%%%%%
\medskip

\noindent\textbf{Proof of the  Proposition \ref{pp1}}: 
%%%%%%%%%%%%%%%%%%%%%%%%%%%%%%%%%%%%%%%%%%%%%%%%%%%%%%%%%%%%%%%%%%%%%%%%%%%%%%%%%%%%%%%%%%%%%%%%%%%%%%%%%%%%%%%%%%%%%%%%%%%%%%%%%%%%%%%%%
(1) Recall that the uniform convergence result uses an argument developed in \cite{Big2}. Fix a compact $K \subset {\mathcal J}$. By  Lemma~\ref{l2} we can fix a compact neighborhood $K'$ of $K$ and $\widetilde p_{K'}>1$ such that
$$  \displaystyle\sup_{q\in K'}\E \Big( \big( \displaystyle\sum_{i=1}^N e^{\left\langle q| X_i  \right\rangle  } \big)^{\tilde p_{K'}} \Big) < \infty. $$
 By Lemma~\ref{l1}, we can fix  $1<p_K\le \min (2,\tilde p_{K'})$ such that $\sup_{q\in K}\phi(p_K,q)<1$. Then for each  $q\in K$, there exists a neighborhood $V_{q} \subset \C^d$ of $q$, whose projection to $\R^d$ is contained in $K'$, and  such that  for all $u\in T$ and  $z\in V_{q}$,  the random variable
$$W_{z}(u)=\frac{e^{ \left\langle z | X_u \right\rangle}}{\E \Big(\displaystyle\sum_{i=1}^N e^{  \left\langle z | X_i \right\rangle} \Big)} $$
is well defined,  and we have 
 $$\displaystyle\sup_{z\in V_{q}} \phi (p_{K}, z) < 1,$$
 where for all $z,z'\in \C^d$ we set $ \left\langle z | z' \right\rangle = \displaystyle\sum_{i=1}^d z_i \bar{z_i}$, and  
 $$
  \phi (p_{K}, z)= \frac{\E\Big (\sum_{i=1}^N |e^{  \left\langle z | X_i \right\rangle}|^{p_K}\Big )}{ \Big |\E \big(\displaystyle\sum_{i=1}^N e^{  \left\langle z | X_i \right\rangle} \big)\Big |^{p_K}}.
  $$
By extracting a finite covering of  $K$ from $\displaystyle\bigcup_{q\in K} V_{q}$, we find a neighborhood  $V\subset \C^d$ of $K$ such that 
$$\displaystyle \sup_{z\in V} \phi (p_{K}, z) < 1.$$
%We notice that this inequality then holds for all $1<p\le p_K$. 

Since the projection of $V$ to $\R^d$ is included in $ K'$ and the mapping $z\mapsto  \E \big(\sum_{i=1}^N e^{  \left\langle z | X_i \right\rangle} \big)$ is continuous and does not vanish on $V$, by considering a smaller neighborhood of $K$ included in $V$ if necessary, we can assume that  
$$
A_V=\sup_{z\in V} \E\Big (\big | \sum_{i=1}^N e^{  \left\langle z | X_i \right\rangle} \big |^{p_K}\Big ) \Big |\E \big(\sum_{i=1}^N e^{  \left\langle z | X_i \right\rangle} \big)\Big |^{-p_K} + 1<\infty.
$$

Now, for  $u\in T$, we define the analytic extension to $V$ of $Y_n(u,q)$ given by 
\begin{align*}
Y_n(u,z)&=\sum_{v\in T_n(u)}W_z(u\cdot v_1)\cdots W_z(u\cdot v_1\cdots v_n)\\
&=\E \big(\sum_{i=1}^N e^{  \left\langle z | X_i \right\rangle} \big)^{-n}\sum_{v\in T_n(u)} e^{ \left\langle z| S_{|u|+n}X(uv)-S_{|u|}(u)\right\rangle}.
\end{align*}
We denote also  $Y_n(\emptyset, z )$ by $Y_n(z)$. Now, applying Lemma \ref{ll1}, with $V_i=e^{\left\langle z | X_i \right\rangle}/\E \big(\displaystyle\sum_{j=1}^N e^{  \left\langle z | X_j \right\rangle} \big)$ and $M=1$, we get      
\begin{multline*}
\E\big(\left| Y_n(z)-Y_{n-1}(z) \right|^{p_{K}} \big)\\  \leq C_{p_K} \Big (\E\big(  \sum_{i=1}^{N}| V_{i} |^{p_K} \big)\Big )^{n-1} \Big (\E\big( \sum_{i=1}^{N}| V_{i} |^{p_K} \big) +   \E \big(  |\sum_{i=1}^{N} V_i |^{p_K}\big)+1\Big ).
\end{multline*}
Notice that $\E\Big( \sum_{i=1}^{N}| V_{i} |^{p_K} \Big) = \phi (p_K, z)$. Then, 
\begin{align*}
&\E \big(\left| Y_n(z)-Y_{n-1}(z) \right|^{p_{K}}\big) \\
&\le C_{p_{K}} \displaystyle \sup_{z\in V} \phi (p_{K}, z)^n +  C_{p_{K}} A_V \displaystyle \sup_{z\in V} \phi (p_{K}, z)^{n-1}.
\end{align*}

With probability $1$, the functions  $z \in V \mapsto Y_n(z), n\geq 0$,  are  analytic. Fix a closed polydisc    $D(z_{0}, 2\rho) \subset  V$. Theorem  (\ref{Cauchy})  gives 
$$\displaystyle\sup_{z\in D(z_{0},\rho)} \left|Y_n(z)-Y_{n-1} (z)\right| \leq 2^d \int_{[0,1]^d} \left|Y_n(\zeta(\theta))-Y_{n-1}(\zeta(\theta))\right| d\theta,$$
where, for $\theta = (\theta_1, \cdots, \theta_d) \in [0, 1]^d$,  $$\zeta(\theta) = z_0 + 2\rho (e^{i2\pi \theta_1}, \cdots, e^{i2\pi \theta_d}) \;\text{and} \; d\theta = d\theta_1\cdots d\theta_d.$$
Furthermore Jensen's inequality and Fubini's Theorem give 
\begin {eqnarray*}
\E \big(\displaystyle\sup_{z\in D(z_0,\rho)} & &\left|Y_n(z)-Y_{n-1} (z)\right| ^{p_{K}} \big) \\
 &\leq& \E \Big( (2^d \int_{[0,1]^d} \left|Y_n(\zeta(\theta))-Y_{n-1}(\zeta(\theta))\right| d\theta)^{p_{K}} \Big)\\
&\leq& 2^{d p_{K}} \E \Big(\int_{[0,1]^d} \left|Y_n(\zeta(\theta))-Y_{n-1}(\zeta(\theta))\right|^{p_{K}} d\theta \Big)\\
&\leq& 2^{dp_{K}} \int_{[0,1]^d} \E \left|Y_n(\zeta(\theta))-Y_{n-1}(\zeta(\theta))\right|^{p_{K}} d\theta \\
&\leq&  2^{dp_{K}} C_{p_{K}}  \displaystyle \sup_{z\in V} \phi (p_{K}, z)^n +  C_{p_{K}} \displaystyle \sup_{z\in V} \phi (p_{K}, z)^{n-1}  A_V.
\end{eqnarray*}
Since $\displaystyle \sup_{z\in V} \phi (p_{K}, z) < 1$, it follows that $\displaystyle\sum_{n\geq 1}\big\|\displaystyle\sup_{z\in D(z_0,\rho)} \left|Y_n(z)-Y_{n-1} (z)\right| \big\|_{p_K} <\infty $. This implies, $z\mapsto Y_n(z)$ converge uniformly, almost surely and in $L^{p_K}$ norm over the  compact $D(z_{0}, \rho)$ to a limit $z\mapsto Y(z)$. This also implies that   
$$\Big\| \displaystyle\sup_{z\in P(z_{0},\rho)} Y(z) \Big\| _{p_{K}} < \infty. $$
Since $K$ can be covered by finitely many  such polydiscs $D(z_0,\rho)$    we get the uniform  convergence, almost surely and in $L^{p_K}$ norm,  of the sequence  $(q\in K\mapsto Y_n(q))_{n\geq1}$ to $q\in K\mapsto Y(q)$. Moreover, since  ${\mathcal J}$ can be covered by a countable  union of such compact $K$ we get the simultaneous convergence for all   $q\in {\mathcal J}$. The same holds simultaneously for all the function $q\in {\mathcal J} \mapsto Y_n(u,q)$, $u\in \bigcup_{n\ge 0}\N_+^n$, because $ \bigcup_{n\ge 0}\N_+^n$ is countable. \\ 

To finish the proof of  Proposition \ref{pp1}(1), we must show that with  probability $1$,  $q\in K \mapsto Y(q)$ does not vanish. Without loss of generality we can suppose that $K=[0,1]^d$. If $I$ is a dyadic closed subcube of $[0,1]^d$, we denote by $E_{I}$ the event $\{ \exists \; q \in I : Y(q)=0 \}$. Let  $I_0,I_1,\cdots,I_{2^d-1}$ stand for the $2^d$ dyadic subcubes of $I$ in the next generation. The event  $E_{I}$ being a tail event of probability $0$ or $1$, if we suppose  that  $P(E_{I})=1$, there exists  $j \in \{0, 1,\cdots,2^d-1\}$ such that $P(E_{I_{j}}) = 1$. Suppose now that $P(E_{K})=1$. The previous  remark allows to construct a decreasing sequence   $(I(n))_{n\geq 0}$ of dyadic subscubes of $K$ such that $P(E_{I(n)})=1$. Let $q_0$ be the unique element of $\displaystyle \cap_{n\geq 0} I(n)$. Since  $q \mapsto Y(q)$ is continuous  we have  $P(Y(q_0)=0)=1$, which contradicts the fact that $(Y_n(q_0))_{n\geq 1}$ converge to $Y(q_0)$ in $L^{1}$.

\medskip

\noindent
(2) It is a consequence of the branching property 
$$ Y_{n+1}(u,q)= \displaystyle\sum_{i=1}^{N} e^{\langle q|X_{ui}\rangle -\tilde{P}(q)} Y_n(ui,q).$$

%%%%%%%%%%%%%%%%%%%%%%%%%%%%%%%%%%%%%%%%%%%%%%%%%%%%%%%%%%%%%%%%%%%%%%%%%%%%%%%%%%%%%%%%%%%%%%%%%%%%%%%%%%%%%%%%%%%%%%%%%%%%%%%%%%%%%%%%%
\medskip

\noindent
\textbf{Proof of  Proposition \ref {pp2}}:
%%%%%%%%%%%%%%%%%%%%%%%%%%%%%%%%%%%%%%%%%%%%%%%%%%%%%%%%%%%%%%%%%%%%%%%%%%%%%%%%%%%%%%%%%%%%%%%%%%%%%%%%%%%%%%%%%%%%%%%%%%%%%%%%%%%%%%%%%
Let $K$ be a compact subset of $ {\mathcal J}$.  For all $q\in K$,  there exists a compact neighborhood $\Lambda$  of the origin such that $ \{q+\lambda : q\in K, \lambda \in \Lambda\} \subset {\mathcal J} $. Let $R=\{q+\lambda : q\in K, \lambda \in \Lambda\}$.  For $ q\in K$ and $\lambda \in \Lambda$ we define
%\begin{align*}
%F_n(q,\lambda)&=\displaystyle\sum_{u\in T_n} e^{\left\langle q+\lambda | S_n(u)\right\rangle -n\widetilde P(q)} Y(q,u)\\
%\text{and}\quad \quad \quad \quad \quad \quad &\\ 
$$Z_n(q,\lambda)=  \displaystyle\sum_{u\in T_n} e^{\left\langle q+\lambda | S_n(u)\right\rangle -n\widetilde P(q+\lambda)} Y(u, q).$$
%\end{align*}

As in the proof of Proposition \ref{pp1}, we can find $p_R\in (1,2]$ and a neighborhood $V \times V_\Lambda\subset \C^d \times \C^d$ of $K \times \Lambda$ such that  the function 
$$Z_n(z, z') =  \Big(  \E\big(\displaystyle\sum_{i=1}^N e^{\left\langle z+z' | X_i\right\rangle} \big) \Big)^{-n}\displaystyle\sum_{u\in T_n} e^{\left\langle z+z' | S_n(u)\right\rangle} Y(u, z),$$ 
% $$Z_n(z,z')= \frac{F_n(z,z')}{\E(F_n(z,z'))}$$ 
are well defined on $ V\times  V_\Lambda$,  and 
$$
\begin{cases}
\sup_{z'\in V_\Lambda}\sup_{z\in V} \phi( p_{R}, z+z') < 1,\\
\displaystyle
A_{V\times V_\Lambda}=\sup_{(z,z')\in V\times V_\Lambda} \E\Big (\big | \sum_{i=1}^N e^{  \left\langle z+z' | X_i \right\rangle} \big |^{p_R}\Big ) \Big |\E \big(\sum_{i=1}^N e^{  \left\langle z+z' | X_i \right\rangle} \big)\Big |^{-p_R} + 1<\infty.
\end{cases}
$$
%(where $p_R$ is chosen such that $1 < p_R \le \min(2, \tilde p_R)$ and $\tilde p_R$ is the real number associated to the compact $R$ defined in the Lemma \ref{l2}.\\
Suppose that  for each $(z_0,z'_0)\in V\times V_\Lambda$ and $\rho>0$ such that $D(z_0,2\rho)\times D(z'_0,2\rho)\subset V\times V_\Lambda$ we have 
\begin{equation}\label{eq0}
\displaystyle\sum_{n\geq 1}\E \Big(\sup_{(z,z')\in D(z_0,\rho)\times D(z'_0,\rho)} \left| Z_n(z,z')-Z_{n-1}(z,z') \right|^{p_R}\Big) < \infty.
\end{equation} 
then,  with probability $1$, $(z,z')\mapsto Z_n(z,z')$ converges uniformly on $D(z_0,\rho)\times D(z'_0,\rho)$ to a limit $Z(z,z')$, whose restriction to $K\times \Lambda$ can be shown to be positive, in the same way as $Y(\cdot)$ was show to be positive. Since $K\times \Lambda$ can be covered by  finitely many polydiscs of the previous form $D(z_0,\rho)\times D(z'_0,\rho)$,  we get the almost sure uniform convergence of $Z_n(q,\lambda)$ over $K\times\Lambda$ to $Z(q,\lambda)>0$, hence the almost sure uniform convergence of $\frac{1}{n}\log(Z_n(q,\lambda))$ to $0$ over $K\times\Lambda$.  Then the conclusion comes from the fact that,  for $(q, \lambda)\in K \times  \Lambda$, one has 
$$ Z_n(q, \lambda) = \frac{\exp\big(nL_n(q, \lambda)\big)}{\exp\big(n \widetilde P(q+\lambda) - n \widetilde P(q)\big)},$$
indeed, \begin{eqnarray*}
L_n(q,\lambda) &=& \frac{1}{n} \log \int_{\partial T} \exp \big(\left\langle \lambda | S_n(t) \right\rangle \big) d\mu_q(t)\\
&=&\frac{1}{n} \log   \displaystyle\sum_{u\in T_n} \exp(\left\langle \lambda | S_n(u)\right\rangle) \mu_q([u]) \\ 
&=& \frac{1}{n} \log   \displaystyle\sum_{u\in T_n} \exp(\left\langle q+\lambda | S_n(u)\right\rangle -n\widetilde P(q)) Y(u, q).\\ 
\end{eqnarray*}

Now we prove  (\ref{eq0}). Given $(z,z')\in V\times V_\Lambda$, applying Lemma \ref{ll1} with $V_i= e^{\left\langle z+z' | X_i \right\rangle}/\E \big(\sum_{j=1}^N e^{  \left\langle z+z' | X_j \right\rangle} \big)$  and $M=Y(z)$  we get
\begin{multline*}
\E \big (\left|  Z_n(z,z') - Z_{n-1}(z,z')\right|^{p_R} \big )\\\leq C_{p_R} \E(| Y(z) |^{p_R}) (\phi(p_R, z+z' )^n  + A_{V\times V_\Lambda} \phi(p_R, z+z' )^{n-1}).
\end{multline*}

\medskip
For $\tilde z=(z,z')\in V\times V_\Lambda$ and $n\ge 1$ let  $M_n(\tilde z)= Z_n(z,z') - Z_{n-1}(z,z') $. With probability $1$ the functions $\tilde z\in V\times V_\Lambda \mapsto M_n(\tilde z)$, $n\ge 1$, are analytic.  Fix a closed  polydisc   $D(\tilde z_0,2\rho)\subset V\times V_\Lambda$ with  $\rho > 0 $. Theorem (\ref{Cauchy}) gives
$$\displaystyle\sup_{\tilde z\in D(\tilde z_{0},\rho)} | M_n(\tilde z) | \leq 2^{2d} \int_{[0,1]^{2d}}  | M_n(\zeta(\theta) ) |d\theta,$$
where, for $\theta=(\theta_1, \cdots, \theta_{2d}) \in [0, 1]^{2d}$,
 $$\zeta(\theta) = \tilde z_0 + 2\rho (e^{i2\pi \theta_1}, \cdots, e^{i2\pi \theta_{2d}}) \;\text{and} \; d\theta = d\theta_1\cdots d\theta_{2d}.$$
Furthermore  Jensen's inequality and Fubini's Theorem give 
\begin {eqnarray*}
&&\E \big(\displaystyle\sup_{\tilde z\in D(\tilde z_0,\rho)} \left|M_n(\tilde z)\right| ^{p_R} \big) \\ 
&\leq& \E \big( (2^{2d} \int_{[0,1]^{2d}} \left|M_n(\zeta(\theta))\right| d\theta)^{p_R} \big)\\
&\leq& 2^{ 2d p_R} \E \big(\int_{[0,1]^{2d}} \left|M_n(\zeta(\theta))\right|^{p_R} d\theta \big)\\
&\leq& 2^{2dp_R} \int_{[0,1]^{2d}} \E \left|M_n(\zeta(\theta))\right|^{p_R} d\theta \\
&\leq&  2^{2dp_R} C_{p_R} \E \big( \displaystyle \sup_{z\in V}| Y(z)|^{p_R} \big) \\
&&\cdot \Big (\displaystyle \sup_{(z,z')\in V\times V_\Lambda} \phi (p_R, z+z' )^{n}+A_{V\times V_\Lambda}\sup_{(z,z')\in V\times V_\Lambda} \phi (p_R, z +z' )^{n-1}\Big )\Big ).
\end{eqnarray*}
Since $\sup_{(z,z')\in V\times V_\Lambda} \phi (p_R, z+z' )<1$, we obtain  the conclusion (\ref{eq0}). %This imply that, with probability $1$, $(z,z') \mapsto Z_n(z,z')$ converge uniformly over the  compact $D(z_{0}, \rho)$ to a limit $Z(z,z')$. 
%Since $V\times \widetilde \Lambda$ can be covered by finitely many  such polydisk we get the uniform  convergence of the sequence  $(Z_n(q,\lambda))_{n\geq1}$ to $Z(q,\lambda)$ for all  $q\in K$ and $\lambda\in K_\Lambda$.\\

\medskip
%%%%%%%%%%%%%%%%%%%%%%%%%%%%%%%%%%%%%%%%%%%%%%%%%%%%%%%%%%%%%%%%%%%%%%%%%%%%%%%%%%%%%%%%%%%%%%%%%%%%%%%%%%%%%%%%%%%%%%%%%%%%%%%%%%%%%%%%%%%%%%%%%%
\noindent
\textbf{Proof of the Proposition \ref{pp3}}
%%%%%%%%%%%%%%%%%%%%%%%%%%%%%%%%%%%%%%%%%%%%%%%%%%%%%%%%%%%%%%%%%%%%%%%%%%%%%%%%%%%%%%%%%%%%%%%%%%%%%%%%%%%%%%%%%%%%%%%%%%%%%%%%%%%%%%%%%%%%%%%%%%
Let  $K$ be a compact subset of ${\mathcal J}$. For  $a > 1$, $q\in K$ and $n\geq1$, we set 
$$E_{n,a}^+= \big\{t\in \partial T : Y(t_{|n}, q ) > a^{n} \big \},$$
and 
$$E_{n,a}^-= \big\{t\in \partial T : Y(t_{|n}, q ) < a^{-n} \big \}.$$
It is sufficient to  show that  for  $E \in \{E_{n,a}^+, E_{n,a}^-\}$,
\begin{equation} \label {eq4}
\E \big (\displaystyle\sup_{q\in K} \displaystyle\sum_{n\geq 1} \mu_{q}(E) \big )< \infty.
\end{equation}
Indeed, if this holds, then  with probability $1$,  for each $q\in K$ and $E \in \{E_{n,a}^+, E_{n,a}^-\}$ $\sum_{n\geq 1} \mu_{q}(E) < \infty$, hence by the Borel-Cantelli lemma, for $\mu_{q}$-almost every  $t\in \partial T$, if  $n$ is big enough we have
$$-\log a \leq   \displaystyle\liminf_{n\rightarrow\infty}\frac{1}{n}\log Y(t_{|n},q) \leq \displaystyle\limsup_{n\rightarrow\infty}\frac{1}{n}\log Y(t_{|n},q) \leq \log a.$$
Letting $a$ tend to $1$ along a countable sequence yields the result. 

Let us  prove (\ref{eq4}) for $E = E_{n,a}^+$ (the case $E = E_{n,a}^-$ is similar). At first we have,   
\begin{eqnarray*}
  \sup_{q\in K} \mu_{q}(E_{n,a}^{+})&=& \sup_{q\in K}  \displaystyle\sum_{u\in T_n} \mu_{q}([u]) \mathbf{1}_{\{ Y(u,q) > a^n \}}\\
  &=& \sup_{q\in K}  \displaystyle\sum_{u\in T_n}   e^{\left\langle q | S_n(u) \right\rangle} e^{-n\tilde{P}(q)} Y(u,q) \mathbf{1}_{\{ Y(u,q) > a^n \}}\\
  &\leq& \sup_{q\in K}  \displaystyle\sum_{u\in T_n} e^{\left\langle q| S_n(u)\right\rangle} e^{-n\tilde{P}(q)} (Y(u,q))^{1+\nu} a^{-n \nu},  \\
  &\leq& \sup_{q\in K}  \displaystyle\sum_{u\in T_n} e^{\left\langle q| S_n(u)\right\rangle} e^{-n\tilde{P}(q)} M(u)^{1+\nu} a^{-n \nu},  \\
\end{eqnarray*}
where $M(u) = \displaystyle\sup_{q\in K} Y(u, q)$ and  $\nu > 0$ is an arbitrary parameter.
For $q \in K$ and $\nu > 0$, we set  $H_n(q,\nu) = \displaystyle\sum_{u\in T_n} e^{\left\langle q| S_n(u)\right\rangle} e^{-n\tilde{P}(q)} M(u)^{1+\nu}  a^{-n \nu}$.\\

\medskip

For $q \in K$, we have $\E \Big( \displaystyle\sum_{i=1}^N e^{\left\langle q| X_i\right\rangle} \Big) = e^{\widetilde P(q)} \neq 0$. Then, there exists  a neighborhood $U_K\subset \C^d$  of $K$ such that $\E \Big( \displaystyle\sum_{i=1}^N e^{\left\langle z| X_i\right\rangle} \Big)  \neq 0$ for all $z\in U_K$.

 \begin{lem}
Fix $a>1$. For $z\in U_K$ and  $\nu > 0$, let  $$H_n(z,\nu) =  \E \Big( \displaystyle\sum_{i=1}^N e^{\left\langle z| X_i\right\rangle} \Big)^{-n} \displaystyle\sum_{u\in T_n} e^{\left\langle z| S_n(u)\right\rangle}  M(u)^{1+\nu}  a^{-n \nu}.$$ There exists  a neighborhood $V\subset \C^d$  of $K$ and a positive constant $C_K$  such that,  for all $z \in V$, for all integer $n \ge 1$,  
 \begin{equation} 
\E\big(\big |  H_n(z, p_K-1) \big | \big)  \leq C_K a^{-n(p_K-1)/2},
 \end{equation}
 where $p_K$ provided by Proposition (\ref{pp1}).
 \end{lem}
 \begin{proof}
For $z\in U_K$ and $\nu > 0$, let
$$
\widetilde H_1(z,\nu)= \big |\E\big(   \displaystyle\sum_{i=1}^N  e^{\left\langle z| X_i \right\rangle}    \big)\big |^{-1} \E\big(   \displaystyle\sum_{i=1}^N \big |  e^{\left\langle z| X_i \right\rangle}   \big | \big) \;  a^{- \nu}.
$$
Let $q\in K$. Since $\E(\widetilde H_1(q,\nu)) = a^{-\nu}$, there exists a neighborhood $V_q \subset U_K$ of  $q$ such that for all $z\in V_q$  we have  $\E \Big( \big | \widetilde H_1(z, \nu)  \big | \Big) \le a^{-\nu/2}$.  By extracting a finite covering of $K$ from $\displaystyle \bigcup_{q\in K} V_q$, we find a neighborhood $V \subset U_K$ of $K$ such that  $\E \Big( \big | \widetilde H_1(z, \nu)  \big | \Big) \le a^{-\nu/2}$ for all $z \in V$. Therefore,  
\begin{eqnarray*} 
\E\big(\big |  H_n(z,\nu) \big | \big)& = & \big | \E\big(   \displaystyle\sum_{i=1}^N e^{\left\langle z| X_i \right\rangle}    \big)\big |^{-n} \E\big(  \big | \displaystyle\sum_{u\in T_n} e^{\left\langle z| S_{n}X(u)\right\rangle}   M(u)^{1+\nu} \big | \big)  \;a^{-n \nu} \\
& \le & \big | \E\big(   \displaystyle\sum_{i=1}^N e^{\left\langle z| X_i \right\rangle}    \big)\big |^{-n} \E\big( \displaystyle\sum_{u\in T_n}  \big | e^{\left\langle z| S_{n}X(u)\right\rangle}   \big |  M(u)^{1+\nu}\big)  \;a^{-n \nu} \\
\end{eqnarray*}
By Proposition (\ref{pp1}) there exists $p_K \in (1, 2]$ such that  for all $u \in \bigcup_{n\ge 0} \N^n_+$, $\E\big(M(u)^{p_K}\big) = \E\big(M(\emptyset)^{p_K}\big)= C_K < \infty.$ Take $\nu=p_{K}-1$ in the last calculation, it follows, from the independence of  $\sigma \big (\{ (X_{u1},\cdots,X_{uN(u)}), u\in T_{n-1} \}\big )$ and $\sigma \big (\{ Y(u,\cdot), u\in T_{n}\} \big )$ for all $n\geq 1$, that
\begin{eqnarray*} 
\E\Big(\Big |  H_n(z, p_K -1) \Big | \Big)&\le&  \Big |\E\Big(   \displaystyle\sum_{i=1}^N  e^{\left\langle z| X_i \right\rangle}    \Big)\Big |^{-n} \E\Big(   \displaystyle\sum_{i=1}^N \Big |  e^{\left\langle z| X_i \right\rangle}   \Big | \Big)^n \;  C_K a^{-n (p_K-1)}\\
 &=& C_K \E \Big( \big | \widetilde H_1(z, p_K-1)  \big | \Big)^n \le  C_K a^{-n(p_K-1)/2},
\end{eqnarray*}
then the Lemma is now proved.
\end{proof}

\medskip

 With probability $1$, the functions $z\in V \longmapsto H_n(z,\nu)$ are analytic. Fix  a closed  polydisc   $D(z_{0},2\rho) \subset V$, $\rho > 0 $ such that $D(z_{0}, 2\rho)\subset V$. Theorem (\ref{Cauchy}) gives  
$$\displaystyle\sup_{z\in D(z_{0},\rho)} \big | H_n(z, p_K-1) \big | \leq 2^d \int_{[0,1]^d}  \big | H_n( \zeta(\theta), p_K-1 ) \big | d\theta,$$
where, for $\theta= (\theta_1, \cdots, \theta_d) \in [0, 1]^d$,
$$\zeta(\theta) = z_0 + 2\rho (e^{i2\pi \theta_1}, \cdots, e^{i2\pi \theta_d}) \;\text{and} \; d\theta = d\theta_1\cdots d\theta_d.$$
Furthermore Fubini's Theorem  gives 
\begin {eqnarray*}
\E \big(\displaystyle\sup_{z\in D(z_0,\rho)} \left|H_n(z, p_K -1)\right| \big)  &\leq& \E \big( 2^{d} \int_{[0,1]^d} \left|H_n(\zeta(\theta), p_K -1)\right| d\theta \big)\\
&\leq& 2^{d} \int_{[0,1]^d} \E \left| H_n( \zeta(\theta), p_K -1)\right| d\theta \\
&\leq&  2^{d}C_K a^{-n(p_K -1)/2}.
\end{eqnarray*}
Since $a>1$ and $p_K -1 > 0 $, we get (\ref {eq4}).\\

%%%%%%%%%%%%%%%%%%%%%%%%%%%%%%%%%%%%%%%%%%%%%%%%%%%%%%%%%%%%%%%%%%%%%%%%%%%%%%%%%%%%%%%%%%%%%%%%%%%%%%%%%%%%%%%%%%%%%%%%%%%%%%%%%%%%%%%%%%%%%%%%%%
\section{Remarks }\label{remarks}
%%%%%%%%%%%%%%%%%%%%%%%%%%%%%%%%%%%%%%%%%%%%%%%%%%%%%%%%%%%%%%%%%%%%%%%
%%%%%%%%%%%%%%%%%%%%%%%%%%%%%%%%%%%%%%%%%%%%%%%%%%%%%%%%%%%%%%%%%%%%%%%
\begin{enumerate}
\item To estimate the dimension of the measure $\mu_q$, we could have introduced, the logarithmic generating functions 
$$ \tilde L_n(q,s)=\frac{1}{n} \log \int_{\partial T}  \mu_q(x_{|n})^s d\mu_q(x) , \;\;\;(q\in J,\ s\in \R),$$ 
and studied their convergence in the same way as $L_n(q,s)$ was studied in Proposition~\ref{pp2}. However, we would have had to find an analytic extension of the mapping $q\mapsto Y(q)^{1+s}$, almost surely in a deterministic neighborhood of any compact subset of ${\mathcal J}$ in order to apply the technique using Cauchy formula. It turns out that the existence of such an extension is not clear, but assuming its existence, the same approach as in the proof of Corollary~\ref{cc1} would give the Hausdorff dimension of $\mu_q$.  If we only seek for a result valid for each $q\in {\mathcal J}$ almost surely, then it is not hard to get the almost sure uniform convergence of  $s\mapsto \tilde L_n(q,s)$ in a compact neighborhood of $0$ towards $s\mapsto \widetilde P(q(1+s))-(1+s)\widetilde P (q)$, and the same approach as that of Corollary~\ref{cc1} yields the dimension of $\mu_q$.

%%%%%%%%%%%%%%%%%%%%%%%%%%%%%%%%%%%%%%%%%%%%%%%%%%%%%%%%%%%%%%%%%%%%%%%%%
%%%%%%%%%%%%%%%%%%%%%%%%%%%%%%%%%%%%%%%%%%%%%%%%%%%%%%%%%%%%%%%%%%%%%%%%
\medskip

\item The method used in this paper is not a direct extension of that used in \cite{jul3} for the case $d=1$ on homogeneous trees. Indeed, in \cite{jul3} the complex extension is used to build simultaneously the measures $\mu_q$, but the proof that, uniformly in $q$, $\mu_q$ is carried by $E(P'(q))$ and has a Hausdorff dimension $P(q)-qP'(q)$ uses a real analysis method, which  seems hard to extend in general when $d\ge 2$. Indeed, such an extension should use the injection of  Sobolev spaces of the form $W^{1,p} (U)$ ($U$ an open subset of $\R^d$) into a space of H\"older continuous functions \cite[p. 28]{NE} to  control the uniform convergence of series like $\sum_{n\ge 1} Z_n(q,\lambda)$ in the proof of Proposition~\ref{pp2}; however, such an inclusion requires $p> d\ge 2$, so that we leave the range of orders of moments for which we have nice controls thanks to  Lemma~\ref{ll2}. 

\medskip

\item Our assumptions can be relaxed as follows. We could assume that $\widetilde P$ is finite over a neighborhood $V$ of $0$, consider ${\mathcal J}_V=\{q\in V: \widetilde P(q)-\langle q|\nabla\widetilde P(q)\rangle >0\} \cap \Omega^1$, 
%and suppose that there exists a continuous function $q\in J_V\mapsto p_q\in (1,\infty)$ such that $ \E\Big (\big (\sum_{i=1}^N e^{\langle q|X_i\rangle}\big )^{p_q}\Big )<\infty $ for all $q\in J_V$.  
Then  the same conclusions as in Theorem~\ref{tt} hold with $I=\{\nabla \widetilde P(q): q\in {\mathcal J}_V\}$. 

\medskip

\item Suppose that $\widetilde P$ is finite over $\R^d$, and without loss of generality that  it is strictly convex. Then $I$ is open, and one can show that $\overline I =\{\alpha\in\R^d:\widetilde P^*(\alpha)\ge 0\}$. Even if $J\subset  \Omega^1$ so that we achieved the multifractal analysis on $I$, it remains the non trivial question of the Hausdorff dimension of $E (\alpha)$ for $\alpha\in \partial {I}$. This problem cannot be solved by the method used  in this paper. In dimension $1$, this boundary consists of two points, and the question has been partially solved  in \cite{jul3}  and completly in \cite{jul5} by buiding a suitable random measure (not of Mandelbrot type) on $E (\alpha)$. It would be easy to adapt the same method  to show here  that  if $\alpha\in\partial I$ is of the form $\nabla\widetilde P(q)$ with $P^*(\alpha)=0$, or if $\alpha\in \partial I$ and there exists $q_0\in\R^d$ such that  $\alpha=\displaystyle\lim_{\lambda\ra\infty}\nabla\widetilde P(\lambda q_0)$, then we have $E(\alpha)\neq\emptyset$ and $\dim E(\alpha)=\widetilde P^*(\alpha)$. In \cite{AB}, a new approach unifying the cases $\alpha\in I$ and $\alpha\in\partial I$ is used to proved that almost surely, for all $\alpha\in \overline I$ we have  $E(\alpha)\neq\emptyset$ and $\dim E(\alpha)=\widetilde P^*(\alpha)$, without any reference to $\Omega^1$. 

\medskip 

\item  It is worth mentioning that a simple consequence of the proof of the previous result is the following large deviation property, which could also be deduced from~\cite{Big2}: With probability 1, 
 \begin{equation*} \label{LD}
  \forall\ \alpha\in I,\ \lim_{\epsilon\to 0}\lim_{n\to\infty}\frac{1}{n}\log \#\{u\in T_n: \|S_n (u)-n\alpha\|\le n\epsilon\}=\widetilde P^*(\alpha).
 \end{equation*}
 Indeed this property essentially follows from the fact that for all $\beta\in B(\alpha,\epsilon)$,  $\{[u]: \ u\in T_n,\  \|S_n (u)-n\alpha\|\le n\epsilon\}$ form, for $n$ large enough,  a sequence of coverings of diameter tending to 0 of a  subset $E$ of $E(\beta)$ with $\dim E=\dim E(\beta)=\widetilde P^*(\beta)$.  Hence $\liminf_{n\to\infty}\frac{1}{n}\log \#\{u\in T_n: \|S_n (u)-n\alpha\|\le n\epsilon\}\ge \sup_{\beta\in B(\alpha,\epsilon)} \widetilde P^*(\beta)$; the other inequality $\limsup_{n\to\infty}\frac{1}{n}\log \#\{u\in T_n: \|S_n (u)-n\alpha\|\le n\epsilon\}\le \sup_{\beta\in B(\alpha,\epsilon)} \widetilde P^*(\beta)$ follows from Chernoff inequalities. 

\end{enumerate}

\section{Appendix}

 \subsection{Cauchy formula in several variables}
 Let us recall the Cauchy formula for holomorphic functions in several variables.
\begin{defi}
Let $d \ge  1$, a subset  $D$ of $\C^d$ is an open  polydisc if there exist  open discs  $D_{1},...,D_{d}$ of $\C$ such that  $D=D_{1}\times...\times D_{d}$. If we denote by $\zeta_{j}$ the centre of  $D_{j}$, then $\zeta=(\zeta_{1},...,\zeta_{d})$ is the centre of  $D$ and if  $r_{j}$ is the radius of  $D_{j}$ then  $r=(r_{1},...,r_{d})$ is the  multiradius of $D$. The set  $\partial D=\partial D_{1} \times...\times \partial D_{d}$ is the distinguished boundary of  $D$. We denote by  $D(\zeta,r)$ the  polydisc with  center $\zeta$ and   radius $r$. 

 Let $D=D(\zeta,r)$ be a polydisc  of $\C^ d $ and  $g \in C(\partial P)$ a continuous function on $\partial D$. We define the integral of  $g$ on  $\partial D$ as 
 $$\int_{\partial D} g(\zeta) d\zeta_{1}...d\zeta_{d}=(2 i \pi)^d r_{1}...r_{d} \int_{[0,1]^d} g(\zeta(\theta)) e^{i 2\pi\theta_{1}}...e^{i2\pi\theta_{d}} d\theta_{1}...d\theta_{d},  $$
 where $\zeta(\theta)= (\zeta_{1}(\theta),...,\zeta_{d}(\theta))$ and $\zeta_{j}(\theta)=\zeta_{j} + r_{j} e^{i 2\pi\theta_{j}}$ for $j=1,...,d.$
   \end{defi}
  \begin{thm}\label {Cauchy}
 Let $D=D(a,r)$ be  polydisc in $\C^d$ with a multiradius whose components are positive,  and $f$ be a holomorphic function in a neiborhood of  $D$.  Then, for all $z\in P$
 $$f(z)=\frac{1}{(2 i \pi)^d} \int_{\partial D} \frac{f(\zeta) d\zeta_{1}...d\zeta_{d}}{(\zeta_{1}-z_{1})...(\zeta_{d}-z_{d})}. $$
It follows that 
\begin{equation}\label{eeq1}
 \displaystyle\sup_{z\in D(a,r/2)} \left|f(z)\right| \leq 2^{d} \int_{[0,1]^d} \left|f(\zeta(\theta))\right| d\theta_{1}...d\theta_{d} 
\end{equation} \end{thm}

%%%%%%%%%%%%%%%%%%%%%%%%%%%%%%%%%%%%%%%%%%%%%%%%%%%%%%%%%%%%%%%%%%%%%%%%%%%%%%%%%%%%%%%%%%%%%%%%%%%%%%%%%%%%%%%%%%%%%%%%%%%%%%%%%% 
%%%%%%%%%%%%%%%%%%%%%%%%%%%%%%%%%%%%%%%%%%%%%%%%%%%%%%%%%%%%%%%%%%%%%%%%
%%%%%%%%%%%%%%%%%%%%%%%%%%%%%%%%%%%%%%%%%%%%%%%%%%%%%%%%%%%%%%%%%%%%%%%%
\subsection{Mass distribution principle }
%%%%%%%%%%%%%%%%%%%%%%%%%%%%%%%%%%%%%%%%%%%%%%%%%%%%%%%%%%%%%%%%%%%%%%%%
%%%%%%%%%%%%%%%%%%%%%%%%%%%%%%%%%%%%%%%%%%%%%%%%%%%%%%%%%%%%%%%%%%%%%%%%

\begin{thm}\label{Bill}\cite[Theorem 4.2]{Falconer}
Let $\nu $ be a positive and finite Borel probability measure on a  compact metric space $(X,d)$. Assume that $M \subseteq X$ is a Borel set such that  $\nu(M) > 0$ and 
$$ M \subseteq \big\{ t\in X, \displaystyle\liminf_{r\ra 0^+ }\frac{\log \nu(B(t,r))}{\log r } \geq \delta  \big\}.$$
Then the Hausdorff dimension of $M$ is bounded from below by $\delta$. 
 \end{thm}
\subsection{Subgradient of convexe function}
Let $f: \R^d \lra \bar \R$, and $x\in \R^d$. A vector $\xi \in \R^d$ is said to be subgradient of $f$ at $x$ if
$$\forall y\in \R^d, \;\;\; f(y) \geq f(x) +\left\langle \xi | y-x \right\rangle.$$
The set of all subgradient of $f$ at $x$ is denoted by $\partial f (x)$. 
\begin{prop}\label{pppp1} \cite {Roc}
If $f$ is convex and differentiable at $x$, then $\partial f(x)= \{ \nabla f(x)\}$.
\end{prop}

\medskip

\textbf{Acknowledgement}

\medskip

The author would like to thank  professor Julien Barral for his interesting comments and many valuable suggestions on this work. 

\medskip

The author would like to thank  the referee  for his interesting comments which contributed to improve the paper.

\end{document}